\documentclass[12pt,reqno]{amsart}

\usepackage{multirow}
\usepackage{hhline}

\usepackage{mathtools}
\usepackage{times}
\usepackage[T1]{fontenc}
\usepackage{mathrsfs}
\usepackage{latexsym}
\usepackage[dvips]{graphics}
\usepackage[titletoc, title]{appendix}
\usepackage{amsmath,amsfonts,amsthm,amssymb,amscd}
\usepackage{color}
\usepackage{hyperref}
\usepackage{amsmath}

\usepackage{color}
\usepackage{breakurl}

\usepackage{comment}
\newcommand{\bburl}[1]{\textcolor{blue}{\url{#1}}}

\newtheorem{thm}{Theorem}[section]

\newtheorem{cor}[thm]{Corollary}

\newtheorem{lem}[thm]{Lemma}
\newtheorem{prop}[thm]{Proposition}

\newtheorem{defi}[thm]{Definition}
\newtheorem{rek}[thm]{Remark}

\DeclareMathOperator{\supp}{supp}

\DeclareMathOperator{\sgn}{sgn}
\numberwithin{equation}{section}

\begin{document}

\title[Weighted Greedy Bases]{On weighted greedy-type bases}

\author{H\`ung Vi\d{\^e}t Chu}

\email{\textcolor{blue}{\href{mailto:hungchu2@illinois.edu}{hungchu2@illinois.edu}}}
\address{Department of Mathematics, University of Illinois at Urbana-Champaign, Urbana, IL 61820, USA}

\begin{abstract} 
In this paper, we study weights for the Thresholding Greedy Algorithm (TGA). While previous work focused on sequential weights $\varsigma = (s_n)_{n\in\mathbb{N}}$ on each positive integer, we study a more general weight $\omega = (w_A)_{A\subset\mathbb{N}}$ on each set $A\subset \mathbb{N}$. 
We define and characterize $\omega$-(almost) greedy bases. Furthermore, we leverage existing results to show that there exists an $\omega$-greedy unconditional basis that is not $\varsigma$-almost greedy for any weight sequence $\varsigma$. Last but not least, we show the equivalence between $\omega$-semi-greedy bases and $\omega$-almost greedy bases when $\omega$ is a so-called structured weight, thus considerably extending the equivalence previously known to hold for sequential weights. 
\end{abstract}

\subjclass[2020]{41A65; 46B15}

\keywords{Thresholding greedy algorithm, greedy, almost greedy, semi-greedy, partially greedy, weight}

\thanks{The author is thankful to Timur Oikhberg for helpful feedback on earlier drafts of this paper.}

\date{\today}

\maketitle

\section{Introduction}
\subsection{Background}

Let $(\mathbb{X}, \|\cdot\|)$ be a Banach space over the field $\mathbb{K} = \mathbb{R}$ or $\mathbb{C}$ with a semi-normalized Schauder basis $\mathcal{B} = (e_n)_{n=1}^\infty$ satisfying 
\begin{equation}\label{ed1} 0\ <\ c_1\ :=\ \inf_n\|e_n\|\ \le\ \sup_n\|e_n\| \ =:\ c_2<\infty.\end{equation}
Let $(e_n^*)_{n=1}^\infty\subset\mathbb{X}^*$ be the biorthogonal functionals such that $e_n^*(e_m) = \delta_{n,m}$. Every $x\in\mathbb{X}$ can be uniquely written as the series $\sum_{n=1}^\infty e_n^*(x)e_n$. Recall that the partial sum operators, defined as $S_m(x) = \sum_{n=1}^m e_n^*(x)e_n$, are uniformly bounded. We let $\mathbf K_b := \sup_{m}\|S_m\|$. It is easy to verify that for a semi-normalized basis, the corresponding biorthogonal functionals are also semi-normalized, i.e., 
\begin{equation}\label{ed2} 0\ <\ c^*_1\ :=\ \inf_n\|e^*_n\|\ \le\ \sup_n\|e^*_n\| \ =:\ c^*_2<\infty.\end{equation}

In 1999, Konyagin and Temlyakov \cite{KT1} introduced the \textit{Thresholding Greedy Algorithm} (TGA) to approximate each vector $x$ using finite linear combinations of basis vectors. In particular, for each $x\in\mathbb{X}$ and $m\in\mathbb{N}$, a set $\Lambda_m(x)$ is called a \textit{greedy set} of order $m$ of $x$ if $|\Lambda_m(x)| = m$ and
$$\min_{n\in \Lambda_m(x)}|e_n^*(x)|\ \ge\ \max_{n\notin \Lambda_m(x)}|e_n^*(x)|.$$
For $x\in\mathbb{X}$, the TGA produces a sequence of approximating greedy sums 
$(G_m(x))_{m=1}^\infty$, where $G_m(x):= \sum_{n\in\Lambda_m}e_n^*(x)e_n$. Here $G_m(x)$ depends on $\Lambda_m(x)$. A basis is \textit{quasi-greedy} if there exists $C\ge 1$ such that
\begin{equation}\label{e1}\|x-G_m(x)\|\ \le\ C\|x\|, \forall x\in\mathbb{X}, \forall m\in\mathbb{N}, \forall \Lambda_m(x).\end{equation}
The least constant $C$ satisfying \eqref{e1} is denoted by $\mathbf{C}_\ell$, and we say $\mathcal{B}$ is $\mathbf C_\ell$-suppression quasi-greedy.
A basis is \textit{greedy} if the TGA gives essentially the best approximation, i.e., there exists a constant $C\ge 1$ such that 
$$\|x-G_m(x)\|\ \le\ C\sigma_m(x),\forall x\in\mathbb{X}, \forall m\in\mathbb{N}, \forall \Lambda_m(x),$$
where 
$$\sigma_m(x)\ :=\ \inf\left\{\left\|x-\sum_{n\in A}a_ne_n\right\|\,:\, A\subset\mathbb{N}, |A| = m, a_n\in \mathbb{K}\right\}.$$
A basis is \textit{almost greedy} if the TGA gives essentially the best projection approximation, i.e., there exists a constant $C\ge 1$ such that 
$$\|x-G_m(x)\|\ \le\ C\widetilde{\sigma}_m(x),\forall x\in\mathbb{X}, \forall m\in\mathbb{N},\forall \Lambda_m(x),$$
where 
$$\widetilde{\sigma}_m(x)\ :=\ \inf\left\{\left\|x-\sum_{n\in A}e_n^*(x)e_n\right\|\,:\, A\subset\mathbb{N}, |A| = m\right\}.$$
A beautiful theorem of Konyagin and Temlyakov \cite{KT1} characterizes greedy bases as being unconditional and democratic (defined later.) In the same spirit, Dilworth et al. \cite{DKKT} characterized almost greedy bases as being quasi-greedy and democratic. 

As a variant of (almost) greedy bases, one can introduce the weighted version, where a weight sequence $\varsigma = (s_n)_{n=1}^\infty \in (0,\infty)^{\mathbb{N}}$ is involved. Given a set $A\subset \mathbb{N}$, the weight of $A$ is $s(A):= \sum_{n\in A} s_n$. For $\alpha\ge 0$, we define the error $\sigma_{\alpha}^\varsigma(x)$ and $\widetilde{\sigma}_\alpha^\varsigma(x)$
\begin{align*}
    \sigma_\alpha^\varsigma(x)&\ :=\ \inf\left\{\left\|x-\sum_{n\in A}a_ne_n\right\|\,:\, |A|<\infty, s(A)\le \alpha, a_n\in \mathbb{K}\right\},\mbox{ and }\\
    \widetilde{\sigma}_\alpha^\varsigma(x)&\ :=\ \inf\left\{\left\|x-\sum_{n\in A}e_n^*(x)e_n\right\|\,:\, |A|<\infty, s(A)\le \alpha\right\}.
\end{align*}

\begin{defi}\normalfont
A basis $\mathcal{B}$ is
\begin{enumerate}
\item $\varsigma$-greedy if there exists a constant $C\ge 1$ such that 
\begin{equation}\label{e3}\|x-G_m(x)\|\ \le\ C\sigma_{s(\Lambda_m(x))}^\varsigma(x), \forall x\in \mathbb{X}, \forall m\in\mathbb{N}, \forall \Lambda_m(x).\end{equation}
\item $\varsigma$-almost greedy if there exists a constant $C\ge 1$ such that 
\begin{equation}\label{e4}\|x-G_m(x)\|\ \le\ C\widetilde{\sigma}_{s(\Lambda_m(x))}^\varsigma(x), \forall x\in \mathbb{X}, \forall m\in\mathbb{N}, \forall \Lambda_m(x).\end{equation}
\end{enumerate}
\end{defi}

Weighted greedy-type bases have received much attention and witnessed progresses in various directions: see \cite{BL,  BDKOW, B1, DKTW, K}. Specifically, \cite{BDKOW, DKTW} characterized weighted greedy and weighted almost greedy bases; \cite{BL, B1, DKTW} studied weighted weak semi-greedy bases and weighted semi-greedy bases; \cite{K} investigated weighted partially greedy and weighted reverse partially greedy bases. It follows trivially from \cite[Remark 3.3 and Theorem 3.11]{K} that there exists a basis that is weighted partially greedy but is not partially greedy. Furthermore, \cite[Remark 4.10]{DKTW} gives an example of a basis that is weighted greedy but is not almost greedy. 

\subsection{Weights on sets and main goals}
This paper generalizes sequential weights $\varsigma$ to weights on sets $\omega$. First, we characterize weighted (almost) greedy bases (Section \ref{gag}). We show that weights on sets are very general in the sense that a basis is unconditional if and only if it is $\omega$-greedy for some $\omega$. Similarly, a basis is quasi-greedy if and only if it is $\omega$-almost greedy for some $\omega$ (see Corollary \ref{cc}.) 

For our next result, it is worth noting that S. J. Dilworth et al. \cite{DKTW} briefly discussed a generalization of sequential weights $\varsigma$, denoted by $\nu$. Here $\nu$ satisfies 
$$\nu(\emptyset) \ =\ 0 \mbox{ and }\nu(A)\le \nu(B)\Longrightarrow \nu(A\backslash B)\le \nu(B\backslash A).$$ 
In \cite[Remark 2.7]{DKTW}, the authors provided an example of a weight $\nu$ on subsets of $\mathbb{N}$ that cannot be obtained by any sequential weight $\varsigma$. However, it was not known if a $\nu$-weighted basis must be $\varsigma$-weighted for some sequential weight $\varsigma$. Motivated by this, we leverage recent results to obtain an example of an $\omega$-greedy basis that is not $\varsigma$-greedy for any $\varsigma$ (Theorem \ref{m3}.) 

Last but not least, we study $\omega$-semi-greedy and $\omega$-partially greedy bases. One notable result is the equivalence between $\omega$-semi-greedy and $\omega$-almost greedy bases (Theorem \ref{m8}), which considerably extends the same equivalence known to hold for sequential weights $\varsigma$.

To prepare for the next section, we give a formal definition of a general weight $\omega$ on sets and $\omega$-(almost) greedy bases. 

\begin{defi}\label{dw}\normalfont Let $\mathcal{P}(\mathbb{N})$ be the power set of $\mathbb{N}$.
A weight on sets is a nonnegative function $\omega: \mathcal{P}(\mathbb{N})\rightarrow [0,\infty]$ such that
\begin{itemize}
    \item $w(\emptyset) = 0$,
    \item $w(A) \in (0,\infty]$ for each nonempty $A\subset\mathbb{N}$.
\end{itemize}
\end{defi}

For $B\subset \mathbb{N}$, define
$$\sigma^{\omega}_B(x)\ :=\ \inf\left\{\left\|x-\sum_{n\in A}a_ne_n\right\|\,:\, A\in\mathbb{N^{<\infty}}, w(A\backslash B)\le w(B\backslash A),  a_n\in \mathbb{K}\right\},$$
where $\mathbb{N^{<\infty}}$ is the set of all finite subsets of $\mathbb{N}$.

The following definition of $\omega$-greedy bases generalizes the classical weighted bases (see \cite[Definition 1.1]{BDKOW}.)

\begin{defi}\label{dg}\normalfont
A basis $\mathcal{B}$ is $\omega$-greedy if there exists a constant $C\ge 1$ such that for all $x\in \mathbb{X}$, $m\in\mathbb{N}$, and $\Lambda_m(x)$,
$$\|x-G_m(x)\|\ \le\ C\sigma^{\omega}_{\Lambda_m(x)}(x).$$
The least constant $C$ is denoted by $\mathbf C^\omega_g$.
\end{defi}

In a similar manner, we define $\omega$-almost greedy. For $B\subset \mathbb{N}$, let
$$\widetilde{\sigma}^{\omega}_B(x)\ :=\ \inf\left\{\left\|x-P_A(x)\right\|\,:\, A\in\mathbb{N^{<\infty}}, w(A\backslash B)\le w(B\backslash A)\right\},$$
where $P_A(x) = \sum_{n\in A}e_n^*(x)e_n$.

\begin{defi}\label{dag}\normalfont
A basis $\mathcal{B}$ is $\omega$-almost greedy if there exists a constant $C\ge 1$ such that for all $x\in \mathbb{X}$, $m\in\mathbb{N}$, and $\Lambda_m(x)$, 
$$\|x-G_m(x)\|\ \le\ C\widetilde{\sigma}^{\omega}_{\Lambda_m(x)}(x).$$
The least constant $C$ is denoted by $\mathbf C^\omega_{al}$.
\end{defi}

\begin{rek}\normalfont
It is easy to check that \eqref{e3} and \eqref{e4} are special cases of Definitions \ref{dg} and \ref{dag}, respectively. 
\end{rek}

\section{Characterizations of $\omega$-(almost) greedy bases}\label{gag}

In order to characterize $\omega$-(almost) greedy bases, we need the notion of unconditionality and $\omega$-Property (A). 

\begin{defi}\normalfont
A basis $\mathcal{B}$ is unconditional if there exists $C\ge 1$ such that
$$\|P_A (x)\|\le C\|x\|,\forall x\in \mathbb{X}, \forall A\subset\mathbb{N}.$$
In this case, we say that $\mathcal B$ is $C$-suppression unconditional. The least such $C$ is denoted by $\mathbf K_s$. For an unconditional basis, there also exists a constant $\mathbf K_u$ such that 
$$\left\|\sum_{n=1}^N a_ne_n\right\|\ \le\ \mathbf K_u\left\|\sum_{n=1}^N b_ne_n\right\|,$$
for all $N\ge 1$ and for all scalars $a_n, b_n$ with $|a_n|\le |b_n|$.
\end{defi}

Let
$$1_A \ =\ \sum_{n\in A}e_n \mbox{ and } 1_{\varepsilon A}\ =\ \sum_{n\in A}\varepsilon_n e_n,$$
where $\varepsilon = (\varepsilon_n)_{n=1}^\infty\subset \mathbb{K}$ and $|\varepsilon_n| = 1$. For $x\in \mathbb{X}$, $\supp(x): = \{n: e_n^*(x)\neq 0\}$, $\|x\|_\infty := \sup_{n}|e_n^*(x)|$, and we write $A \sqcup B\sqcup x$ to indicate that $A, B$, and $\supp(x)$ are pairwise disjoint. Finally, $\sgn(e_n^*(x)) = \begin{cases}1&\mbox{ if }e_n^*(x)=0,\\ e_n^*(x)/|e_n^*(x)|&\mbox{ if }e_n^*(x)\neq 0.\end{cases}$

\begin{defi}\label{dA}\normalfont
A basis $\mathcal{B}$ has $\omega$-Property (A) if there exists $C\ge 1$ such that 
$$\|x+1_{\varepsilon A}\|\ \le\ C\|x+1_{\delta B}\|,$$
for all $A, B\in \mathbb{N}^{<\infty}$ with $w(A)\le w(B)$, signs $(\varepsilon), (\delta)$, and $x\in\mathbb{X}$ with $\|x\|_\infty\le 1$ and $A\sqcup B\sqcup x$. The least constant $C$ is denoted by $\mathbf C^{\omega}_b$.
\end{defi}

\begin{thm}\label{m1}
Let $\mathcal{B}$ be a basis and $\omega$ be a weight on subsets of $\mathbb{N}$. 
\begin{enumerate}
    \item If $\mathcal{B}$ is $\mathbf C_g^\omega$-$\omega$-greedy, then $\mathcal{B}$ is $\mathbf C_g^\omega$-suppression unconditional and satisfies $\mathbf C_g^\omega$-$\omega$-Property (A). 
    \item If $\mathcal{B}$ is $\mathbf K_s$-suppression unconditional and satisfies $\mathbf C^{\omega}_b$-$\omega$-Property (A), then $\mathcal{B}$ is $\mathbf K_s \mathbf C^{\omega}_b$-$\omega$-greedy.
\end{enumerate}
\end{thm}

First, we need an useful reformulation of $\omega$-Property (A). 
\begin{lem}\label{l1}
A basis $\mathcal{B}$ has $\mathbf C_b^\omega$-$\omega$-Property (A) if and only if 
\begin{equation}\label{e2}\|x\|\ \le\ \mathbf C_b^\omega\|x-P_A(x) + 1_{\varepsilon B}\|,\end{equation}
for all $x\in\mathbb{X}$ with $\|x\|_\infty\le 1$, $A, B\in \mathbb{N}^{<\infty}$ with $w(A) \le w(B)$ and $B\cap (A\cup\supp(x)) = \emptyset$, and sign $(\varepsilon)$.
\end{lem}

\begin{proof}
Assume that $\mathcal{B}$ has $\mathbf C_b^\omega$-$\omega$-Property (A). Let $x, A, B, (\varepsilon)$ be chosen as in \eqref{e2}. We have
\begin{align*}\|x\|\ =\ \left\|x-P_A(x) + \sum_{n\in A}e_n^*(x)e_n\right\|&\ \le\ \sup_{(\delta)}\left\|x-P_A(x) + 1_{\delta A}\right\|\\
&\ \le\ \mathbf C_b^\omega\left\|x-P_A(x) + 1_{\varepsilon B}\right\|,
\end{align*}
as desired.

Next, assume that $\mathcal{B}$ satisfies \eqref{e2}. Let $x, A, B, (\varepsilon), (\delta)$ be chosen as in Definition \ref{dA}. Let $y = x+ 1_{\varepsilon A}$. By \eqref{e2},
$$\|x+1_{\varepsilon A}\|  \ =\ \|y\|\ \le\ \mathbf C_b^\omega\|y-P_A(y) + 1_{\delta B}\|\ =\ \mathbf C_b^\omega\|x + 1_{\delta B}\|.$$
This completes our proof.
\end{proof}

\begin{prop}\label{p1} Let $\mathcal B = (e_n)_{n=1}^\infty$ be a $\mathbf K_s$-suppression unconditional basis. Fix $N\in \mathbb{N}$. For any scalars $a_1, \ldots, a_N, b_1, \ldots, b_N$ so that either $a_0 = 0$ or $\sgn(a_n) = \sgn(b_n)$ and $|a_n|\leqslant |b_n|$ for all $1\le n\le N$, we have
$$\left\|\sum_{n=1}^N a_ne_n\right\|\ \le \ \mathbf{K}_s\left\|\sum_{n=1}^N b_n e_n\right\|.$$
\end{prop}

\begin{proof}
See \cite[Proposition 2.1]{AW}.
\end{proof}

\begin{proof}[Proof of Theorem \ref{m1}]
Assume that $\mathcal{B}$ is $\mathbf{C}_g^\omega$-$\omega$-greedy. Let $x\in\mathbb{X}$ and $B\in \mathbb{N}^{<\infty}$. Write 
$$y \ =\ \sum_{n\in B}(\alpha + e_n^*(x))e_n + P_{B^c}(x),$$
where $\alpha$ is chosen sufficiently large such that $B$ is a greedy set of $y$. Then
$$\|P_{B^c}(x)\|\ =\ \|y- G_{|B|}(y)\|\ \le\ \mathbf{C}_g^\omega\sigma^\omega_B(y)\ \le\ \mathbf C_g^\omega\|y-\alpha 1_B\|\ =\ \mathbf C_g^\omega \|x\|.$$
Hence, $\mathcal{B}$ is $\mathbf{C}_g^\omega$-suppression unconditional. Next, we prove $\omega$-Property (A). We choose $x, A, B, (\varepsilon), (\delta)$  as in Definition \ref{dA}. Set $y= x+1_{\varepsilon A} + 1_{\delta B}$. Then $B$ is a greedy set of $y$. We have
$$\|x+1_{\varepsilon A}\| \ =\ \|y - G_{|B|}(y)\|\ \le\ \mathbf{C}_g^\omega\sigma^\omega_B(y)\ \le\ \mathbf{C}_g^\omega\|y- P_A(y)\|\ =\ \mathbf{C}_g^\omega\|x+1_{\delta B}\|.$$
This completes the proof.

Now we assume that $\mathcal{B}$ is $\mathbf K_s$-suppression unconditional and satisfies $\mathbf C_b^\omega$-$\omega$-Property (A). Let $x\in\mathbb{X}$ have a greedy set $A\subset\mathbb{N}$. Let $B\in \mathbb{N}^{<\infty}$ with $w(B\backslash A)\le w(A\backslash B)$. Also, choose arbitrary $(b_n)_{n\in B}\subset\mathbb{K}$. Let $\alpha := \min_{n\in A}|e_n^*(x)|$. By Lemma \ref{l1} and Proposition \ref{p1}, we have
\begin{align*}
    \|x-P_A(x)\|&\ \le\ \mathbf C_b^{\omega}\left\|x-P_A(x)-P_{B\backslash A}(x) + \alpha\sum_{n\in A\backslash B}\sgn(e_n^*(x))e_n\right\|\\
    &\ \le\ \mathbf C_b^{\omega}\left\|P_{(A\cup B)^c}(x) + \alpha\sum_{n\in A\backslash B}\sgn(e_n^*(x))e_n\right\|\\
    &\ \le\ \mathbf K_s\mathbf C_b^{\omega}\left\| P_{(A\cup B)^c}(x) + \sum_{n\in B}(e_n^*(x)-b_n)e_n + P_{A\backslash B}(x)\right\|\\
    &\ =\ \mathbf K_s\mathbf C_b^{\omega}\left\|x-\sum_{n\in B}b_ne_n\right\|.
\end{align*}
This completes our proof that $\mathcal{B}$ is $\mathbf K_s\mathbf C_b^{\omega}$-$\omega$-greedy. 
\end{proof}

When we do not need tight estimates, the notion of $\omega$-disjoint (super)democracy can play the role of $\omega$-Property (A), providing other characterizations of $\omega$-greedy bases. 

\begin{defi}\label{dD}\normalfont
A basis $\mathcal{B}$ is $\omega$-disjoint democratic ($\omega$-disjoint superdemocratic, respectively) if there exists $C\ge 1$ such that
$$\|1_A\| \ \le\  C\|1_B\|,\mbox{ }(\|1_{\varepsilon A}\|\ \le\ C\|1_{\delta B}\|,\mbox{ respectively}),$$
for all $A, B\in \mathbb{N}^{<\infty}$ with $w(A) \le w(B)$, $A\cap B = \emptyset$ and signs $(\varepsilon), (\delta)$. The least constant $C$ is denoted by $\mathbf C_{d, \sqcup}^\omega$ (and $\mathbf C^\omega_{sd, \sqcup}$, respectively.)
\end{defi}

\begin{rek}\normalfont
A basis $\mathcal{B}$ is said to be $\omega$-\textit{(super)democratic} if in Definition \ref{dD}, we drop the requirement $A\cap B = \emptyset$; $\mathcal{B}$ is said to be \textit{(super)democratic} if it is $\omega$-(super)democratic for $\omega$ being the cardinality weight, i.e., $w(A) = |A|, \forall A\subset\mathbb{N}$.
\end{rek}

\begin{thm}\label{m4}
Let $\mathcal{B}$ be a basis and $\omega$ be a weight on subsets of $\mathbb{N}$. The following are equivalent:
\begin{enumerate}
    \item $\mathcal{B}$ is $\omega$-greedy,
    \item $\mathcal{B}$ is unconditional and satisfies $\omega$-Property (A), 
    \item $\mathcal{B}$ is unconditional and $\omega$-disjoint superdemocratic,
    \item $\mathcal{B}$ is unconditional and $\omega$-disjoint democratic. 
\end{enumerate}
\end{thm}

\begin{proof}
By Theorem \ref{m1}, we know that (1) $\Longleftrightarrow$ (2). It follows immediately from definitions that
$\omega$-Property (A) $\Longrightarrow$ $\omega$-disjoint superdemocratic $\Longrightarrow$ $\omega$-disjoint democratic. Hence, (2) $\Longrightarrow$ (3) $\Longrightarrow$ (4). It remains to show that (4) $\Longrightarrow$ (2). Let $x, A, B, (\varepsilon), (\delta)$ be chosen as in Definition \ref{dA}. We have
\begin{align*}\|x+1_{\varepsilon A}\|\ \le\ \|x\| + \|1_{\varepsilon A}\|&\ \le\  \|x\| + \mathbf K_u\|1_A\|\\
&\ \le\  \|x\| + \mathbf K_u\mathbf C^{\omega}_{d, \sqcup}\|1_B\|\\
&\ \le\ \mathbf K_s\|x+1_{\delta B}\| + \mathbf K_u^2\mathbf C^{\omega}_{d, \sqcup}\|x+1_{\delta B}\|\\
&\ =\ (\mathbf K_s+\mathbf K_u^2\mathbf C^{\omega}_{d, \sqcup})\|x+1_{\delta B}\|.
\end{align*}
This completes our proof.
\end{proof}

For $\omega$-almost greedy bases, corresponding results hold. We include the proof of the next theorem in the Appendix.
\begin{thm}\label{m2}
Let $\mathcal{B}$ be a basis and $\omega$ be a weight on subsets of $\mathbb{N}$. 
\begin{enumerate}
    \item If $\mathcal{B}$ is $\mathbf C_{al}^\omega$-$\omega$-almost greedy, then $\mathcal{B}$ is $\mathbf C_{al}^\omega$-suppression quasi-greedy and satisfies $\mathbf C_{al}^\omega$-$\omega$-Property (A). 
    \item If $\mathcal{B}$ is $\mathbf C_\ell$-suppression quasi-greedy and satisfies $\mathbf C^{\omega}_b$-$\omega$-Property (A), then $\mathcal{B}$ is $\mathbf C_\ell \mathbf C^{\omega}_b$-$\omega$-almost greedy.
\end{enumerate}
\end{thm}

\begin{thm}\label{m5}
Let $\mathcal{B}$ be a basis and $\omega$ be a weight on subsets of $\mathbb{N}$. The following are equivalent:
\begin{enumerate}
    \item $\mathcal{B}$ is $\omega$-almost greedy,
    \item $\mathcal{B}$ is quasi-greedy and satisfies $\omega$-Property (A), 
    \item $\mathcal{B}$ is quasi-greedy and $\omega$-disjoint superdemocratic,
    \item $\mathcal{B}$ is quasi-greedy and $\omega$-disjoint democratic. 
\end{enumerate}
\end{thm}

\begin{cor}\label{cc}\begin{enumerate}
    \item A basis $\mathcal{B}$ is unconditional if and only if $\mathcal{B}$ is $\omega$-greedy for some weight $\omega$.
    \item A basis $\mathcal{B}$ is quasi-greedy if and only if $\mathcal{B}$ is $\omega$-almost greedy for some weight $\omega$.
\end{enumerate}
\end{cor}

\begin{proof}
(1) If $\mathcal{B}$ is $\omega$-greedy for some weight $\omega$, then $\mathcal{B}$ is unconditional by Theorem \ref{m1}. Conversely, suppose that $\mathcal{B}$ is unconditional. Define the weight $\omega$ on a set $A$ 
$$\omega(A)\ =\ \begin{cases}\|1_A\| &\mbox{ if }A\mbox{ is finite,}\\ \infty &\mbox{ if }A\mbox{ is infinite}.\end{cases}$$
By Theorem \ref{m4}, it suffices to show that $\mathcal{B}$ is $\omega$-disjoint democratic. This is clearly true since for two finite sets $A, B$ with $\omega(A)\le \omega(B)$, we get $\|1_A\|\le \|1_B\|$ by the definition of $\omega$. 

The proof of (2) is similar to that of (1). 
\end{proof}

\section{A set-weighted-greedy basis that is not sequence-weighted-greedy}\label{exa}

The following theorem provides a necessary condition for a basis to be $\varsigma$-greedy for some weight sequence $\varsigma$. 

\begin{thm}\label{m0}
If a basis $\mathcal{B} = (e_n)_{n=1}^\infty$ is $\varsigma$-(almost) greedy for some weight sequence $\varsigma$, then either $\mathcal{B}$ is (almost) greedy or there exists a subsequence $(e_{n_k})_{k=1}^\infty$ equivalent to the canonical basis of $c_0$.
\end{thm}

Observe that $\varsigma$-Property (A) is $\omega$-Property (A) when the weight $\omega$ on sets is determined by a weight sequence $\varsigma$. Particularly, Property (A) (first introduced in \cite{AW} and later generalized in \cite{DKOSZ}) is $\varsigma$-Property (A) when $\varsigma = (1,1,\ldots)$.

\begin{defi}\normalfont
A basis $\mathcal{B}$ has $\varsigma$-Property (A) if there exists $C\ge 1$ such that 
$$\|x+1_{\varepsilon A}\|\ \le\ C\|x+1_{\delta B}\|,$$
for all $A, B\in \mathbb{N}^{<\infty}$ with $s(A)\le s(B)$, signs $(\varepsilon), (\delta)$, and $x\in\mathbb{X}$ with $\|x\|_\infty\le 1$ and $A\sqcup B\sqcup x$. The least constant $C$ is denoted by $\mathbf C^{\varsigma}_b$.
As a special case, a basis $\mathcal{B}$ is said to have Property (A) if it has $\varsigma$-Property (A) for $\varsigma = (1, 1, \ldots)$.
\end{defi}

\begin{proof}[Proof of Theorem \ref{m0}]
We assume that $\mathcal{B}$ is $\varsigma$-greedy for some weight sequence $\varsigma = (s(n))_{n=1}^\infty$. By \cite[Theorem 4.1]{BDKOW}, $\mathcal{B}$ is unconditional and has $\varsigma$-Property (A). 

If $0 < \inf s(n) \le \sup s(n) < \infty$, then \cite[Proposition 3.5]{BDKOW} implies that $\mathcal{B}$ has Property (A). According to \cite[Theorem 2]{DKOSZ}, we know that $\mathcal{B}$ is greedy. 

If $\sup s(n) = \infty$, then \cite[Proposition 3.10]{BDKOW} states that $\mathcal{B}$ is equivalent to the canonical basis of $c_0$ and thus, is greedy. 

If $\inf s(n) = 0$, then by \cite[Proposition 3.10]{BDKOW}, $(e_n)_{n=1}^\infty$ has a subsequence $(e_{n_k})_{k=1}^\infty$ that is equivalent to the canonical basis of $c_0$. 

The proof of the almost greedy case is similar. 
\end{proof}

We now state the existence of an $\omega$-greedy basis that is not $\varsigma$-almost greedy for any weight sequence $\varsigma$. We can, in particular, require the weight $\omega$ to have a more rigid structure than in Definition \ref{dw}. For conciseness, we let $w_n := w(\{n\})$.

\begin{defi}\label{sw}\normalfont
A structured weight is a nonnegative function $\omega: \mathcal{P}(\mathbb{N})\rightarrow [0,\infty]$ such that
\begin{itemize}
    \item[(a)] $w(\emptyset) = 0$,
    \item[(b)] $w(A) < \infty$ if $|A|<\infty$,
    \item[(c)] $w(A) \in (0,\infty]$ for each nonempty $A\subset\mathbb{N}$,
    \item[(d)] $w(A)\rightarrow 0$ as $\sum_{n\in A}w_n\rightarrow 0$,
    \item[(e)] $w(A)\rightarrow\infty$ as $\sum_{n\in A}w_n\rightarrow\infty$,
    \item[(f)] There exists an arbitrarily large number $N\in\mathbb{N}$ such that there exists an $\varepsilon > 0$ satisfying $w(\{N, n\})-w_{n} > \varepsilon$ for all $n\in \mathbb{N}, n\neq N$. 
\end{itemize}
\end{defi}
Conditions (a), (b), and (c) are almost the same as what we have in Definition \ref{dw}, except that we now require the weight on a finite set to be finite. Conditions (d) and (e) are reasonable. Condition (d) states that the weight on a set approaches $0$ when the sum of weights of its singletons approaches $0$, while (e) states the same condition with $0$ replaced by $\infty$. Throughout this paper, we will specify whether we need structured weights in our results. If we state a result without mentioning structured weights, then the result holds for weights in Definition \ref{dw}. 

\begin{thm}\label{m3}
There exists a basis that is $\omega$-greedy for some structured weight $\omega$ on sets but is not $\varsigma$-almost greedy for any weight sequence $\varsigma$ on positive integers. 
\end{thm}

\begin{proof}
By Theorem \ref{m0}, an unconditional basis that is neither democratic nor has a subsequence equivalent to the canonical basis of $c_0$ is not $\varsigma$-almost greedy on any $\varsigma$. 

Set $P := \{2^{k}: k\ge 1\}$, $a_n = 1/n^{1/2}$ and $b_n = 1/n$ for $n\ge 1$.
Let $\mathbb{X}$ be the completion of $c_{00}$ under the following norm: for $x = (x_1, x_2, \ldots)\in c_{00}$, define 
$$\|x\|\ :=\ \left(\sup_{\sigma}\sum_{i\in P}a_{\sigma(i)}|x_i|\right) + \left(\sup_{\pi}\sum_{i\notin P} b_{\pi(i)}|x_i|\right),$$
where $\sigma: P\rightarrow \mathbb{N}$ and $\pi: \mathbb N\backslash P\rightarrow \mathbb{N}$ are bijections. Let $\mathcal{B} = (e_n)_{n=1}^\infty$
be the canonical basis. Clearly, $\mathcal{B}$ is unconditional and normalized. However, $\mathcal{B}$
is not democratic. Indeed, fix $N\in\mathbb{N}$ and set $A = \{3^1, 3^2, \ldots, 3^{N}\}$, $B = \{2^1, 2^2, \ldots, 2^{N}\}$. We have
$$\|1_A\|\ =\ \sum_{n=1}^N \frac{1}{n} \ \sim\ \ln (N)\mbox{ and }\|1_B\|\ =\ \sum_{n=1}^N\frac{1}{\sqrt{n}} \ \sim\ \sqrt{N}.$$
Since $\|1_B\|/\|1_A\| \sim \sqrt{N}/\ln(N)\rightarrow \infty$ as $N\rightarrow\infty$, we know that $\mathcal{B}$ is not democratic and thus, not almost greedy. 

By the proof of Corollary \ref{cc}, $\mathcal{B}$ is $\omega$-greedy for the following weight $\omega$
$$\omega(A)\ =\ \begin{cases}\|1_A\| &\mbox{ if }A\mbox{ is finite,}\\ \infty &\mbox{ if }A\mbox{ is infinite}.\end{cases}$$
It is easy to check that $\omega$ is a structured weight.

We claim that there is no subsequence of $\mathcal{B} = (e_n)_{n=1}^\infty$ that is equivalent to the canonical basis of $c_0$.
Indeed, pick any subsequence $(e_{n_k})_{k=1}^\infty$ of $\mathcal{B}$. For $N\in \mathbb{N}$, we have 
$$\left\|\sum_{k=1}^N e_{n_k}\right\|\ \ge\ \sum_{n=1}^N\frac{1}{n}\ \sim\ \ln (N).$$
Hence, $(e_{n_k})_{k=1}^\infty$ is not equivalent to the canonical basis of $c_0$.
Therefore, Theorem \ref{m0} and the fact that $\mathcal{B}$ is not almost greedy tell us that $\mathcal{B}$ is not $\varsigma$-almost greedy for any weight sequence $\varsigma$.
\end{proof}

\section{$\omega$-semi-greedy bases}\label{sc}

First, we define the $\omega$-version of the classical semi-greedy bases (first introduced in \cite{DKK}). Corresponding to each greedy set $\Lambda_m(x)$, there is a so-called \textit{Chebyshev greedy sum} of order $m$, denoted by $CG_{m}(x)$, such that
\begin{enumerate}
    \item $\supp(CG_{m}(x))\subset\Lambda_m(x)$ and 
    \item we have $$\|x-CG_{m}(x)\|\ =\ \min \left\{\left\|x-\sum_{n\in\Lambda_m(x)}a_ne_n\right\|\, :\, (a_n)\subset\mathbb{K}\right\}.$$
\end{enumerate}

\begin{defi}\label{dsg}\normalfont
A basis $\mathcal{B}$ is $\omega$-semi-greedy if there exists a constant $C\ge 1$ such that for all $x\in \mathbb{X}$, $m\in\mathbb{N}$, and $\Lambda_m(x)$,
$$\|x-CG_m(x)\|\ \le\ C\sigma^{\omega}_{\Lambda_m(x)}(x).$$
The least constant $C$ is denoted by $\mathbf C^\omega_s$.
\end{defi}

The main goal of this section is to establish the following theorem, which, by Theorem \ref{m3}, is a nontrivial extension of \cite[Theorem 1.10]{B1}.

\begin{thm}\label{m8}
Let $\omega$ be a structured weight. Then $\mathcal{B}$ is $\omega$-semi-greedy if and only if it is $\omega$-almost greedy. 
\end{thm}

\begin{prop}\label{p42}
Let $\mathcal{B}$ be a $\mathbf C^{\omega}_s$-$\omega$-semi-greedy basis, where $\omega$ is structured. 
\begin{enumerate}
    \item Let $B\in\mathbb{N}^{<\infty}$ and $w(B) \le \limsup\limits_{n\rightarrow\infty} w_n$. Then we have
    $$\sup_{(\varepsilon)}\|1_{\varepsilon B}\|\ \le\ 2\mathbf K_b\mathbf C^{\omega}_sc_2,$$
    where $c_2$ is in \eqref{ed1}.
    \item If $\sup_n w_n = \infty$ or $\sum_n w_n< \infty$, then $\mathcal{B}$ is equivalent to the canonical basis of $c_0$.
    \item If $\inf_n w_n = 0$, then $\mathcal{B}$ contains a subsequence equivalent to the canonical basis of $c_0$.
\end{enumerate}
\end{prop}

\begin{rek}\label{r1}\normalfont
The conclusions in Proposition \ref{p42} still hold if our basis $\mathcal{B}$ is $\mathbf C^{\omega}_{sd, \sqcup}$-disjoint superdemocratic. The proof is left for interested readers. 
\end{rek}

\begin{proof}
(1) Pick $B\in\mathbb{N}^{<\infty}$ and $(\varepsilon)$. Choose $N_1 > \max B$ be the number in condition (f) of a structured weight such that there exists $\varepsilon > 0$ satisfying $w(\{N_1, n\})> w_n+\varepsilon$ for all $n\neq N_1$. It follows that
$$\limsup\limits_{n\rightarrow\infty} w(\{N_1, n\}) \ \ge\ \limsup \limits_{n\rightarrow\infty} w_n + \varepsilon\ \ge\ w(B) + \varepsilon.$$
Pick $N_2 > N_1$ such that $w(\{N_1, N_2\}) > w(B)$. (This is possible due to condition (b).) Set $x:= 1_{\varepsilon B} + e_{N_1} + e_{N_2}$. Then $\{N_1, N_2\}$ is a greedy set of $x$. Let $\|x-CG_{2}(x)\| = \|1_{\varepsilon B} + \alpha_1 e_{N_1} + \alpha_2 e_{N_2}\|$ for some $\alpha_1, \alpha_2\in \mathbb{K}$. We have
\begin{align*}
    \|1_{\varepsilon B}\|&\ \le\ \mathbf K_b\left\|1_{\varepsilon B} + \alpha_1 e_{N_1} + \alpha_2 e_{N_2} \right\|\ \le\ \mathbf K_b\mathbf C^{\omega}_s\sigma^{\omega}_{\{N_1, N_2\}}(x)\\
    &\ \le\ \mathbf K_b\mathbf C^{\omega}_s\|e_{N_1} + e_{N_2}\|\ \le\ 2\mathbf K_b\mathbf C^{\omega}_sc_2.
\end{align*}

(2) If $\sup_n w_n = \infty$, then by (1), $\sup_{(\varepsilon)}\|1_{\varepsilon B}\| \le 2\mathbf K_b\mathbf C^{\omega}_sc_2, \forall B\in\mathbb{N}^{<\infty}$. Hence, the basis is equivalent to the canonical basis of $c_0$. If $\sum_{n} w_n <\infty$, then choose $N\in\mathbb{N}$ such that $\sum_{n=N+1}^\infty w_n$ is so small that $w(E) < w_1$ for all $E\subset \mathbb{N}_{\ge N+1}$. This can be done due to condition (d) of $\omega$. We claim that for any $B\in \mathbb{N}^{<\infty}$ and any sign $(\varepsilon)$, we have
$\|1_{\varepsilon B}\| = O(1)$.
Let $B_1 = B\cap [1, N]$ and $B_2 = B\cap [N+1,\infty)$.
Observe that
\begin{align*}
    \|1_{\varepsilon B}\|\ \le\ \|1_{\varepsilon B_1}\| + \|1_{\varepsilon B_2}\|\  \le\ Nc_2 + \|1_{\varepsilon B_2}\|.
\end{align*}
Set $x:= e_1 + 1_{\varepsilon B_2}$. Then $\{1\}$ is a greedy set of $x$. Let $\|x-CG_{1}(x)\| = \|\alpha e_1 + 1_{\varepsilon B_2}\|$ for some $\alpha\in\mathbb{K}$. Since $w(B_2)< w_1$, we have
\begin{align*}
    \|1_{\varepsilon B_2}\|&\ \le\ (\mathbf K_b+1)\left\|\alpha e_1 + 1_{\varepsilon B_2}\right\|\ \le\ (\mathbf K_b+1)\mathbf C^{\omega}_s\sigma^{\omega}_{\{1\}}(x)\\
    &\ \le\ (\mathbf K_b+1)\mathbf C^{\omega}_s\|e_1\|\ \le\ (\mathbf K_b+1)\mathbf C^{\omega}_sc_2.
\end{align*}
This completes our proof that $\|1_{\varepsilon B}\| = O(1)$ and so, $\mathcal{B}$ is equivalent to the canonical basis of $c_0$.

(3) Choose a subsequence $(n_k)_{k=1}^\infty$ such that $\sum_{k=1}^\infty w_{n_k} < \infty$ and apply (2).
\end{proof}

\begin{thm}\label{m10}Let $\omega$ be a structured weight. 
If a basis $\mathcal{B}$ is $\omega$-semi-greedy, then it is quasi-greedy and $\omega$-superdemocratic. 
\end{thm}

\begin{proof}
Suppose that $\sum_{n=1}^\infty w_n < \infty$ or $\sup_n w_n = \infty$. By Proposition \ref{p42}, we know that $\mathcal{B}$ is equivalent to the canonical basis of $c_0$, and the desired conclusion follows trivially. 
For the rest of the proof, let us assume that $\sum_{n=1}^\infty w_n = \infty$ and $\sup_n w_n < \infty$.

\underline{Quasi-greedy}: Let $x\in \mathbb{X}$ with $\|x\|_\infty \le 1$, $|\supp(x)| < \infty$, and a greedy set $\Lambda_m(x)$.

Case 1: $w(\Lambda_m(x)) \le \limsup\limits_{n\rightarrow\infty} w_n$. By Proposition \ref{p42}, we have
$$\sup_{(\varepsilon)}\|1_{\varepsilon \Lambda_m(x)}\|\ \le\ 2\mathbf K_b\mathbf C^{\omega}_sc_2.$$
By norm convexity,
\begin{align*}\|P_{\Lambda_m(x)}(x)\|&\ \le\ \max_{n}|e_n^*(x)|\sup_{(\varepsilon)}\|1_{\varepsilon \Lambda_m(x)}\|\\
&\ \le\ \sup_{n}\|e_n^*\|\|x\|\cdot 2\mathbf K_b\mathbf C^{\omega}_sc_2\ \le\ 2\mathbf K_b\mathbf C^{\omega}_s c_2c_2^*\|x\|,\end{align*}
where $c_2$ and $c_2^*$ are in \eqref{ed1} and \eqref{ed2}, respectively.

Case 2: $w(\Lambda_m(x)) > \limsup \limits_{n\rightarrow\infty} w_n$. We build a finite set $E$ as follows: choose $N > \max\supp(x)$ such that $w_{N} \le w(\Lambda_m(x))$. Let $k$ be the smallest positive integer verifying
$$w(\{N, N+1, \ldots, N+k\})\ \le\ w(\Lambda_m(x)) < w(\{N, N+1, \ldots, N+k, N+k+1\}).$$
We know such $k$ exists due to $\sum_{n}w_n =\infty$ and condition (e) of a structured weight. 
Let $A = \{N, N+1, \ldots, N + k\}$ and $B = A\cup \{N+k+1\}$. Define 
$$y\ :=\ x-P_{\Lambda_m(x)}(x) + \alpha 1_B,$$
where $\alpha := \min_{n\in \Lambda_m(x)}|e_n^*(x)|$. 
Since $B$ is a greedy set of $y$, by $\mathbf C^{\omega}_s$-$\omega$-semi-greediness, there exist $(b_n)_{n\in B}\subset\mathbb{K}$ such that
\begin{align}\label{e40}
    \|x-P_{\Lambda_m(x)}(x)\|&\ \le\ \mathbf K_b\left\|x-P_{\Lambda_m(x)}(x) + \sum_{n\in B}b_n e_n\right\|\ \le\ \mathbf K_b\mathbf C^{\omega}_s\sigma^\omega_{B}(y)\nonumber\\
    &\ \le\ \mathbf K_b\mathbf C^{\omega}_s\|x+\alpha 1_B\|\ \le\ \mathbf K_b\mathbf C^{\omega}_s(\|x\|+\alpha \|1_A\| + \alpha\|e_{N+k+1}\|).
\end{align}
Pick $j\in\Lambda_m(x)$. We have 
\begin{equation}\label{e41}\alpha \|e_{N+k+1}\|\ \le\ \alpha c_2\ \le\ c_2|e_j^*(x)|\ \le\ c_2\|e_j^*\|\|x\|\ \le\ c_2c_2^*\|x\|.\end{equation}
It remains to bound $\alpha \|1_A\|$. Let $z:= x + \alpha 1_A$. Since $\Lambda_m(x)$ is a greedy set of $z$, $\mathbf C^{\omega}_s$-$\omega$-semi-greediness gives $(t_n)_{n\in \Lambda_m(x)}\subset \mathbb{K}$ such that 
\begin{align}\label{e42}
    \|\alpha 1_A\| &\ \le\ (\mathbf K_b+1)\left\|\sum_{n\in \Lambda_m(x)}t_n e_n + P_{\Lambda_m(x)^c}(x)  +\alpha 1_A\right\|\nonumber\\
    &\ \le\ (\mathbf K_b+1)\mathbf C^{\omega}_s\sigma^{\omega}_{\Lambda_m(x)}(z)\ \le\ (\mathbf K_b+1)\mathbf C^{\omega}_s\|x\|.
\end{align}
From \eqref{e40}, \eqref{e41}, and \eqref{e42}, we have shown that 
$$\|x-P_{\Lambda_m(x)}(x)\|\ =\ O(\|x\|).$$
This completes our proof that $\mathcal{B}$ is quasi-greedy.

\underline{$\omega$-superdemocratic}: Let $A, B\in \mathbb{N}^{<\infty}$ with $w(A)\le w(B)$. Pick signs $(\varepsilon), (\delta)$. 

Case 1: $w(A)\le w(B) \le \limsup\limits_{n\rightarrow\infty} w_n$. By Proposition \ref{p42}, we know that 
$$\|1_{\varepsilon A}\|\ \le\ 2\mathbf K_b\mathbf C^{\omega}_sc_2.$$
On the other hand, if $j = \min B$, then 
$$\|1_{\delta B}\|\ \ge\ \|e_j\|/\mathbf K_b\ \ge\ c_1/\mathbf K_b,$$
where $c_1$ is in \eqref{ed1}.
Therefore, 
$$\|1_{\varepsilon A}\|\ \le\ 2\mathbf K^2_b\mathbf C^{\omega}_s\frac{c_2}{c_1}\|1_{\delta B}\|.$$

Case 2: $w(B) > \limsup \limits_{n\rightarrow\infty} w_n$. As when we prove quasi-greediness, choose $E$ and $F = E\cup \{N\}$ such that 
$A\cup B < E < \{N\}$ and $w(E)\le w(B) < w(F)$. Set $x:= 1_{\varepsilon A} + 1_F$. Then $F$ is a greedy set of $x$. By $\mathbf C^{\omega}_s$-$\omega$-semi-greediness, there exist $(a_n)_{n\in F}\subset\mathbb{K}$ such that 
\begin{equation}\label{e44}\|1_{\varepsilon A}\|\ \le\ \mathbf K_b\left\|1_{\varepsilon A} + \sum_{n\in F}a_n e_n\right\|\ \le\ \mathbf K_b\mathbf C^{\omega}_s\sigma^{\omega}_F(x)\ \le\ \mathbf K_b\mathbf C^{\omega}_s\|1_F\|.\end{equation}
Now, let $y = 1_{\delta B} + 1_{E}$. Since $B$ is a greedy set of $y$, by $\mathbf C^{\omega}_s$-$\omega$-semi-greediness, we obtain
\begin{equation}\label{e45}\|1_{E}\|\ \le\ (\mathbf K_b+1)\left\|\sum_{n\in B}b_n e_n + 1_{E}\right\|\ \le\ \mathbf C^{\omega}_s(\mathbf K_b +1)\sigma_B^\omega(y)\ \le\ \mathbf C^{\omega}_s(\mathbf K_b+1)\|1_{\delta B}\|,\end{equation}
for some $(b_n)_{n\in B}\subset \mathbb{K}$. Furthermore, if $u = \min E$, 
\begin{equation}\label{e46}\|1_F\|\ \le\ \|1_E\| + \|e_N\|\ \le\ \|1_E\| + c_2\ \le\ \|1_E\| + \frac{c_2}{c_1}\|e_u\|\ \le\ \left(\frac{c_2}{c_1}\mathbf K_b+1\right)\|1_E\|.\end{equation}
From \eqref{e44}, \eqref{e45}, and \eqref{e46}, we obtain
$$\|1_{\varepsilon A}\|\ \le\ (\mathbf C^{\omega}_s)^2\mathbf K_b(\mathbf K_b+1)\left(\frac{c_2}{c_1}\mathbf K_b + 1\right)\|1_{\delta B}\|.$$
Hence, $\mathcal{B}$ is $\omega$-superdemocratic.
\end{proof}

The proof of the next theorem is similar to that of \cite[Theorem 3.2]{DKK} with obvious modifications, so we move the proof to the Appendix. 

\begin{thm}\label{h1}
If a basis $\mathcal{B}$ is quasi-greedy and $\omega$-disjoint superdemocratic, then it is $\omega$-semi-greedy.  
\end{thm}

\begin{proof}[Proof of Theorem \ref{m8}]
The theorem follows from Theorems \ref{m4}, \ref{m10}, and \ref{h1}.
\end{proof}

\section{$\omega$-partially greedy bases}\label{pgc}
Partially greedy bases were first introduced and characterized in \cite{DKKT} to compare the performance of the TGA to that of the partial sum operators $(S_m)_{m=1}^\infty$. In this section, we characterize $\omega$-partially greedy bases and prove the existence of $\omega$-partially greedy bases that are not $\varsigma$-partially greedy for any sequence weight $\varepsilon$.
For each $m\ge 0$, let $L_m := \{1, 2, \ldots, m\}$. The following is a generalization of \cite[Definition 6.1]{BDKOW} and \cite[Definition 3.4]{B0}, which defines ($\varsigma$-)partial greediness.

\begin{defi}\normalfont
A basis is said to be $\omega$-partially greedy if there exists $C\ge 1$ such that for all $x\in\mathbb{X}$, $m\in\mathbb{N}$, and $\Lambda_m(x)$, we have
$$\|x-G_m(x)\|\ \le\ C\overline{\sigma}^\omega_{\Lambda_m(x)}(x),$$
where 
$$\overline{\sigma}^\omega_{A}(x)\ :=\ \inf\left\{\|x-S_k(x)\|: w(L_k\backslash A)\le w(A\backslash L_k)\right\}.$$ The least such $C$ is denoted by $\mathbf C_p^\omega$.
\end{defi}

We shall characterize $\omega$-partial greediness, generalizing existing characterizations of $\varsigma$-partially greedy bases. In \cite{BBL}, the authors introduce partial symmetry for largest coefficients (PSLC).

\begin{defi}\label{dps}\normalfont
A basis is $C$-$\omega$-PSLC if 
$$\|x+1_{\varepsilon A}\|\ \le\ C\|x+1_{\delta B}\|,$$
for all $x\in \mathbb{X}$ with $\|x\|_\infty\le 1$, for all finite sets $A, B\subset\mathbb{N}$ with $w(A)\le w(B)$ and $ A<\supp(x)\sqcup B$, and for all signs $(\varepsilon), (\delta)$. The least constant $C$ is denoted by $\mathbf C^{\omega}_{pl}$.
\end{defi}

\begin{thm}\label{m9}
Let $\mathcal{B}$ be a basis and $\omega$ be a weight on subsets of $\mathbb{N}$. 
\begin{enumerate}
    \item If $\mathcal{B}$ is $\mathbf C_p^\omega$-$\omega$-partially greedy, then $\mathcal{B}$ is $\mathbf C_p^\omega$-suppression quasi-greedy and is $\mathbf C_p^\omega$-$\omega$-PSLC.
    \item If $\mathcal{B}$ is $\mathbf C_\ell$-suppression quasi-greedy and is $\mathbf C^{\omega}_{pl}$-$\omega$-PSLC, then $\mathcal{B}$ is $\mathbf C_\ell\mathbf C^{\omega}_{pl}$-$\omega$-partially greedy.
\end{enumerate}
\end{thm}

\begin{proof}
    Similar to the proof of Theorem \ref{m2}.
\end{proof}

\begin{defi}\label{dcs}\normalfont
A basis $\mathcal{B}$ is $\omega$-conservative ($\omega$-superconservative, respectively) if there exists $C\ge 1$ such that
$$\|1_A\| \ \le\  C\|1_B\|,\mbox{ }(\|1_{\varepsilon A}\|\ \le\ C\|1_{\delta B}\|,\mbox{ respectively}),$$
for all $A, B\in \mathbb{N}^{<\infty}$ with $A<B$, $w(A) \le w(B)$, and signs $(\varepsilon), (\delta)$.
\end{defi}

We have the following equivalences, whose proof is similar to that of Theorem \ref{m4} and thus, is left for interested readers. 

\begin{thm}\label{m20}
Let $\mathcal{B}$ be a basis and $\omega$ be a weight on subsets of $\mathbb{N}$. The following are equivalent:
\begin{enumerate}
    \item $\mathcal{B}$ is $\omega$-partially greedy,
    \item $\mathcal{B}$ is quasi-greedy and is $\omega$-PSLC, 
    \item $\mathcal{B}$ is quasi-greedy and $\omega$- superconservative.
    \item $\mathcal{B}$ is quasi-greedy and $\omega$- conservative.
\end{enumerate}
\end{thm}

The following is an analog of Theorem \ref{m0}.
\begin{thm}\label{m-1}
If a basis $\mathcal{B} = (e_n)_{n=1}^\infty$ is $\varsigma$-partially greedy for some weight sequence $\varsigma = (s(n))_{n=1}^\infty$ with $\inf s(n) > 0$, then $\mathcal{B}$ is partially greedy.
\end{thm}

\begin{proof}
We assume that $\mathcal{B}$ is $\varsigma$-partially greedy for some weight sequence $\varsigma = (s(n))_{n=1}^\infty$. By \cite[Theorem 6.4]{BDKOW}, $\mathcal{B}$ is quasi-greedy and is $\varsigma$-conservative. 

If $0 < \inf s(n) \le \sup s(n) < \infty$, then \cite[Proposition 4.5]{K} implies that $\mathcal{B}$ is conservative. According to \cite[Theorem 3.4]{DKKT}, $\mathcal{B}$ is partially greedy. 

If $\sup s(n) = \infty$, then \cite[Proposition 4.1]{K} states that $\mathcal{B}$ is equivalent to the canonical basis of $c_0$ and thus, is greedy. 
\end{proof}

\begin{thm}\label{m21}
There exists a Schauder basis that is $\omega$-partially greedy for some structured weight $\omega$ on sets but is not $\varsigma$-partially greedy for any weight sequence $\varsigma = (s(n))_{n=1}^\infty$ with $\inf s(n) > 0$. 
\end{thm}

\begin{proof}
The basis $\mathcal{B}$ in Section \ref{exa} is not conservative. To see this, simply pick $A = \{2, 2^2, \ldots, 2^N\}$ and $B = \{3^{N+1}, \ldots, 3^{2N}\}$. We have $\|1_A\|/\|1_B\| \sim \sqrt{N}/\ln(N)\rightarrow\infty$ as $N\rightarrow\infty$. Hence, $\mathcal{B}$ is not partially greedy due to \cite[Theorem 3.4]{DKKT}. 
Applying Theorem \ref{m-1}, we obtain the desired conclusion. 
\end{proof}

\begin{rek}\normalfont
Theorem \ref{m21} is sharp in the sense that we cannot drop the requirement $\inf s(n) > 0$. Indeed, Khurana \cite{K} characterized $\varsigma$-partially greedy bases by quasi-greediness and the so-called $\varsigma$-left-Property (A). By \cite[Remark 3.3]{K}, any basis trivially satisfies $\varsigma$-left-Property (A) with $\varsigma = (s(n))_{n=1}^\infty = (2^{-n})_{n=1}^\infty$. Hence, if we have an $\omega$-partially greedy, it is quasi-greedy by Theorem \ref{m9} and has $\varsigma$-left-Property (A) for $s(n) = 2^{-n}$. Therefore, the basis is automatically $\varsigma$-partially greedy. 
\end{rek}

\section{Questions and discussion}\label{que}
We list several open questions for future research. 

\begin{enumerate}
    \item[Q1] We show that for a structured weight $\omega$, a basis is $\omega$-almost greedy if and only if it is $\omega$-semi-greedy. Does the result hold for a larger class of weights? 
    \item[Q2] For weights in Definition \ref{dw}, is an $\omega$-disjoint superdemocratic basis also $\omega$-superdemocratic? If not, what minimal condition(s) to put on $\omega$ so that the two properties are equivalent.  
\end{enumerate}

For the second question, we know that for a structured weight, an $\omega$-disjoint superdemocratic is $\omega$-superdemocratic.

\begin{prop}\label{p50}
For a structured weight $\omega$, a basis $\mathcal{B}$ is $\omega$-superdemocratic if and only if $\mathcal{B}$ is $\omega$-disjoint superdemocratic. 
\end{prop}

\begin{proof}
Assume that $\mathcal{B}$ is $\omega$-disjoint superdemocratic. Let $A, B\in\mathbb{N}^{<\infty}$ with $w(A) \le w(B)$. Pick signs $(\varepsilon), (\delta)$.
If $\sum_{n=1}^\infty w_n < \infty$ or $\sup_n w_n = \infty$, then by Proposition \ref{p42} and Remark \ref{r1}, $\mathcal{B}$ is equivalent to the canonical basis of $c_0$, and the desired conclusion follows trivially. 
For the rest of the proof, we assume that $\sum_{n=1}^\infty w_n = \infty$ and $\sup_n w_n < \infty$.

Case 1: $w(A)\le \limsup\limits_{n\rightarrow\infty} w_n$. We proceed as in the proof of Theorem \ref{m10} to show $\|1_{\varepsilon A}\|\lesssim \|1_{\delta B}\|$.

Case 2: $w(A) > \limsup\limits_{n\rightarrow\infty} w_n$. Choose $E$ and $F = E\cup \{N\}$ such that 
$A\cup B < E < \{N\}$ and $w(E)\le w(A) < w(F)$.
By $\mathbf C^\omega_{sd,\sqcup}$-$\omega$-disjoint superdemocracy and \eqref{e46}, we have
$$\|1_{\varepsilon A}\|\ \le\ \mathbf C^\omega_{sd,\sqcup}\|1_F\|\ \le\ \mathbf C^\omega_{sd,\sqcup}\left(\frac{c_2}{c_1}\mathbf K_b+1\right)\|1_E\|\ \le\ (\mathbf C^\omega_{sd,\sqcup})^2\left(\frac{c_2}{c_1}\mathbf K_b+1\right)\|1_{\delta B}\|.$$
Therefore, $\mathcal{B}$ is $\omega$-superdemocratic. 
\end{proof}

\section{Appendix}

\subsection{Proof of Theorem \ref{m2}}
The key input is the uniform boundedness of the truncation function. For each $\alpha > 0$, we define the truncation function $T_\alpha$ as follows: for  $b\in\mathbb{K}$,
$$T_{\alpha}(b)\ =\ \begin{cases}\sgn(b)\alpha, &\mbox{ if }|b| > \alpha,\\ b, &\mbox{ if }|b|\le \alpha.\end{cases}$$
We define the truncation operator $T_\alpha: \mathbb{X}\rightarrow \mathbb{X}$ as
$$T_{\alpha}(x)\ =\ \sum_{n=1}^\infty T_\alpha(e_n^*(x))e_n \ =\ \alpha 1_{\varepsilon \Gamma_{\alpha}(x)}+ P_{\Gamma_\alpha^c(x)}(x),$$
where $\Gamma_\alpha(x) = \{n: |e_n^*(x)| > \alpha\}$ and $\varepsilon_n = \sgn(e_n^*(x))$ for all $n\in \Gamma_\alpha(x)$. The operator $T_\alpha$ is well-defined as $|\Gamma_\alpha(x)| < \infty$ for all $\alpha > 0$ and $x\in \mathbb{X}$.

\begin{thm}\label{bto}\cite[Lemma 2.5]{BBG} Let $\mathcal{B}$ be $\mathbf{C}_\ell$-suppression quasi-greedy. Then for any $\alpha > 0$, $\|T_\alpha\|\le \mathbf{C}_\ell$.
\end{thm}

\begin{proof}[Proof of Theorem \ref{m2}]
Assume that $\mathcal{B}$ is $\mathbf C^{\omega}_{al}$-$\omega$-almost greedy. Let $x\in\mathbb{X}$ and $A$ be a greedy set of $x$. We have
$$\|x-P_A(x)\|\ \le\ \mathbf C^{\omega}_{al}\widetilde{\sigma}^{\omega}_A(x)\ \le\ \mathbf C^{\omega}_{al}\|x-P_\emptyset(x)\|\ =\ \mathbf C^{\omega}_{al}\|x\|.$$
The proof of $\mathbf C^{\omega}_{al}$-$\omega$-Property (A) uses the exact argument as in the proof of Theorem \ref{m1}, so we skip it. 

Now assume that $\mathcal{B}$ is $\mathbf C_\ell$-suppression quasi-greedy and satisfies $\mathbf C^{\omega}_b$-$\omega$-Property (A). Let $x\in\mathbb{X}$ and $A$ be a greedy set of $x$. Let $B\in\mathbb{N}^{<\infty}$ such that $w(B\backslash A)\le w(A\backslash B)$. By Lemma \ref{l1} and Theorem \ref{bto}, we have
\begin{align*}
    \|x-P_A(x)\|&\ \le\ \mathbf C_b^{\omega}\left\|x-P_A(x)-P_{B\backslash A}(x) + \alpha\sum_{n\in A\backslash B}\sgn(e_n^*(x))e_n\right\|\\
    &\ \le\ \mathbf C_b^{\omega}\left\|P_{(A\cup B)^c}(x) + \alpha\sum_{n\in A\backslash B}\sgn(e_n^*(x))e_n\right\|\\
    &\ =\ \mathbf C_b^{\omega}\left\|T_\alpha(P_{(A\cup B)^c}(x) + P_{A\backslash B}(x))\right\|\\
    &\ \le\ \mathbf C_\ell\mathbf C_b^{\omega}\left\|x-P_B(x)\right\|.
\end{align*}
This completes our proof that $\mathcal{B}$ is $\mathbf C_\ell\mathbf C_b^{\omega}$-$\omega$-almost greedy.

\subsection{Proof of Theorem \ref{h1}}
\begin{lem}\cite[Lemma 2.3]{BBG}.
Let $\mathcal{B}$ be a $\mathbf C_\ell$-suppression quasi-greedy basis and $x\in \mathbb{X}$. If $A$ is a greedy set of $x$, then
\begin{equation}\label{e16}\min_{n\in A}| e_n^*(x)|\left\|\sum_{n\in A}\varepsilon_ne_n\right\|\ \le\ 2\mathbf C_\ell\left\|x\right\|,\end{equation}
where $\varepsilon_n  = \sgn(e^*_n(x))$.
\end{lem}

\begin{proof}[Proof of Theorem \ref{h1}]
Let us assume that $\mathcal{B}$ is $\mathbf C_\ell$-suppression quasi-greedy and $\mathbf C^\omega_{sd,\sqcup}$-$\omega$-disjoint superdemocratic. Let $x\in\mathbb{X}$ with $|\supp(x)| < \infty$ and $\Lambda_m(x)$ be a greedy set. Fix $\varepsilon > 0$. Let $y = \sum_{n\in A}a_ne_n$, where $A\in \mathbb{N}^{<\infty}, w(A\backslash \Lambda_m(x)) \le w(\Lambda_m(x)\backslash A)$ and $\|x-y\| < \sigma^\omega_{\Lambda_m(x)}(x) + \varepsilon.$ Write $x-y = \sum_{n=1}^\infty b_ne_n$, where $b_n = e_n^*(x)-a_n$ if $n\in A$ and $b_n = e_n^*(x)$ if $n\notin A$. We shall find a vector $w$ with $\supp(w)\subset \Lambda_m(x)$ such that \begin{equation}\label{e48}\|x-w\|\ \le \ \mathbf C_\ell(1 + 4\mathbf C^{\omega}_{sd, \sqcup}\mathbf C_\ell)(\sigma^\omega_{\Lambda_m(x)}(x)+\varepsilon).\end{equation}
Set $\alpha := \max_{n\notin \Lambda_m(x)}|e_n^*(x)|$. If $\alpha = 0$, then choose $w = x$ and we are done. Assume that $\alpha > 0$. Consider the following vector:
\begin{align}z &\ :=\ \sum_{n\in \Lambda_m(x)}T_\alpha (b_n)e_n + P_{ \Lambda_m(x)^c}(x)\label{e35}\\
&\ =\ \sum_{n\in \Lambda_m(x)}T_{\alpha}(b_n)e_n + \sum_{n\notin A\cup\Lambda_m(x)}T_{\alpha}(b_n)e_n + \sum_{n\in A\backslash \Lambda_m(x)}e_n^*(x)e_n\nonumber\\
&\ =\ \sum_{n\notin A\backslash \Lambda_m(x)}T_\alpha(b_n)e_n + \sum_{n\in A\backslash \Lambda_m(x)}e_n^*(x)e_n\nonumber\\
&\ =\ \sum_{n=1}^\infty T_\alpha(b_n)e_n + \sum_{n\in A\backslash \Lambda_m(x)}(e_n^*(x)-T_\alpha(b_n))e_n.\label{e36}
\end{align}
We claim that $x-z$ is a choice for $w$. Indeed, using \eqref{e35}, we know that $\supp(w) = \supp(x-z)\subset \Lambda_m(x)$. By Theorem \ref{bto}, we have
\begin{equation}\label{e37}\left\|\sum_{n=1}^\infty T_\alpha(b_n)e_n\right\|\ \le\ \mathbf{C}_\ell\|x-y\|.\end{equation}
Note that $|e_n^*(x)-T_\alpha(b_n)|\le 2\alpha$ for all $n\in A\backslash \Lambda_m(x)$. Let $\eta = (\sgn(e_n^*(x-y))_{n=1}^\infty$. We have
\begin{align*}
    \left\|\sum_{n\in A\backslash \Lambda_m(x)}(e_n^*(x)-T_\alpha(b_n))e_n\right\|&\ \le\ 2\alpha\sup_{(\delta)}\left\| 1_{\delta A\backslash \Lambda_m(x)}\right\|\\
    &\ \le\ 2\mathbf C^\omega_{sd, \sqcup}\min_{n\in \Lambda_m(x)\backslash A}|e_n^*(x-y)|\|1_{\eta\Lambda_m(x)\backslash A}\|.
\end{align*}
Let $B := \{n: |e_n^*(x-y)|\ge\min_{n\in \Lambda_m(x)\backslash A}|e_n^*(x-y)|\}$. Then $B$ is a greedy set of $x-y$ and $\Lambda_m(x)\backslash A\subset B$. Therefore, we obtain
$$\|1_{\eta\Lambda_m(x)\backslash A}\|\ \le\ \mathbf C_\ell\|1_{\eta B}\|$$
and so, by \eqref{e16},
\begin{align}\label{e47}\left\|\sum_{n\in A\backslash \Lambda_m(x)}(e_n^*(x)-T_\alpha(b_n))e_n\right\|&\ \le\ 2\mathbf C^\omega_{sd, \sqcup}\mathbf C_\ell\min_{n\in \Lambda_m(x)\backslash A}|e_n^*(x-y)|\|1_{\eta B}\|\nonumber\\
&\ \le\ 4\mathbf C^\omega_{sd, \sqcup}\mathbf C^2_\ell\|x-y\|.\end{align}
Using \eqref{e36}, \eqref{e37}, and \eqref{e47}, we obtain \eqref{e48}. Therefore, 
$$\|x-CG_m(x)\| \ \le\ \mathbf C_\ell(1 + 4\mathbf C^{\omega}_{sd, \sqcup}\mathbf C_\ell)(\sigma^\omega_{\Lambda_m(x)}(x)+\varepsilon).$$
Letting $\varepsilon\rightarrow 0$ completes the proof. 
\end{proof}

\end{proof}

\ \\
\end{document}